\documentclass{amsart}

\usepackage{enumerate, amsmath, amsfonts, amssymb, amsthm, thmtools, wasysym, graphics, graphicx, xcolor, frcursive,bbm}
\usepackage[all]{xy}
\usepackage{etex}
\usepackage{lscape}
\usepackage{tikz-cd}
\usepackage{float}
\usepackage{xcolor}

\theoremstyle{plain}
\newtheorem{thm}{Theorem}[section]
\newtheorem{lem}[thm]{Lemma}

\newtheorem{prop}[thm]{Proposition}

\newtheorem{defn}[thm]{Definition}
\newtheorem{rem}[thm]{Remark}

\newtheorem{exmp}[thm]{Example}

\newcommand{\cox}{\textsc{Cox}}
\newcommand{\art}{\textsc{Art}}
\newcommand{\hyp}{\textsc{Hyp}}
\newcommand{\hecke}{\textsc{Hecke}}

\newcounter{georgcomments}

\begin{document}

\title[Cancellative elliptic Artin monoids]{Cancellative elliptic Artin monoids}

\author[Georges Neaime]{Georges Neaime}
\email{gneaime@math.uni-bielefeld.de}
\address{Universität Bielefeld, Postfach 10 01 31, D-33501 Bielefeld}

\begin{abstract}
    We define new presentations for elliptic Artin groups. We also show that the elliptic monoids defined by these presentations are cancellative. This solves the failure of cancellativity for the presentations of elliptic Artin monoids that were introduced before in the literature. Our approach also paves the way to construct Garside structures for elliptic Artin monoids and groups.
\end{abstract}

\date{\today}

\maketitle

\tableofcontents

\section{Introduction}\label{SectionIntroduction}

In the middle of the 80s, Saito \cite{SaitoI} introduced the notion of an elliptic root system inside a vector space $V$ equipped with a symmetric bilinear form $I$ with rank 2 radical. He also introduced the notion of a marking and classified marked elliptic root systems by introducing the notion of elliptic Dynkin diagrams (for example, see Figures~\ref{Fig_Diagram_Saito_D4} and \ref{Fig_Elliptic_E_6,7,8}). 

To each elliptic root system and elliptic diagram $\Gamma$, one can naturally define the notion of an elliptic Weyl group $W(\Gamma)$, defined as the group generated by the reflections associated to the root system. If we consider the hyperbolic extension of the space $(V,I)$, for which the bilinear form has rank $1$ radical, one obtains the notion of a hyperbolic elliptic Weyl group $\hyp{(\Gamma)}$. It fits into a central extension, where the kernel of the natural map from $\hyp{(\Gamma)}$ to $W(\Gamma)$ is the cyclic group generated by a certain power of the so-called Coxeter transformation of $W(\Gamma)$. In 1997, Saito--Takebayashi~\cite{Saito-Takebayashi} defined presentations by generators and relations for $W(\Gamma)$ and $\hyp{(\Gamma)}$. 

To each elliptic Weyl group $W(\Gamma)$, one can define the notion of an elliptic Artin group $\art{(\Gamma)}$ and an elliptic Hecke algebra $\hecke{(\Gamma)}$ (see~\cite{Yamada} for the simply laced cases and \cite{YSaito-Shiota} in general). Similarly to spherical and affine Artin groups, the presentations of elliptic Artin groups are described using homogeneous relations, meaning that every relation preserves the length of words.

By a personal communication with K. Saito, it will be shown in \cite{KSaito-YSaito} that the groups defined by the homogenised relations are isomorphic to the fundamental groups of the complements of elliptic discriminants~\cite{SaitoII} by identifying them with those defined by Looijenga \cite{Looijenga1976, Looijenga1980} and van der Lek \cite{vanderLek-Thesis}.

In the classical theory for spherical Artin groups, a seminal result is that a certain cancellative Artin monoid injects in the Artin group as a lattice. This fits into a general theory called Garside theory (see~\cite{DehornoyBook}). This theory was used in order to prove many important conjectures for Artin groups such as the $K(\pi,1)$ conjecture, and to solve the word problem. 

Concerning the elliptic theory, Saito \cite{Saito2ndHomotopyNonCancellative,SaitoHigherHomotopy} showed that the elliptic Artin monoid defined by these homogeneous relations may no longer be cancellative and therefore, the natural morphism from the monoid to the elliptic Artin group may not be injective. Hence, one is not able to apply Garside theory in this context.

In this paper, we define new presentations for the elliptic Artin groups in the simply laced case and show that the monoids defined by these presentations are cancellative. This is therefore a remarkable step towards the construction of Garside structures for elliptic Artin monoids and groups. 

In order to define the new presentations, we use some techniques that were developed by the author in \cite{Neaime_TildeA} based on the work of Corran--Picantin~\cite{CorranPicantin} and Corran--Lee--Lee~\cite{CorranLeeLee}. These new presentations are described in Propositions~\ref{Prop_New_PresD4} and \ref{Prop_NewPres_Elliptic_En}. It is notable to mention that the presentations described in \cite{Neaime_TildeA} enabled the author to introduce new interval Garside structures for the affine Artin groups of type $\tilde{A}$ based on his earlier discoveries \cite{NeaimeEEN} in the context of complex reflection groups.

Using a combinatorial technique, namely the word reversing technique developed by Dehornoy \cite{DehornoyCompletePosPres}, we are able to show that the elliptic Artin monoids defined by our new presentations are cancellative. This is Theorem~\ref{Thm_Cancellativity}. 

The paper is organised as follows. In Section~\ref{SectionAffineArtinGroupTildeA}, we introduce various presentations for the affine Artin group of type $\tilde{A}$. In particular, we recall in Definition~\ref{Def_PresCorranLeeLee_TildeA} the presentation of Corran--Lee--Lee for $\art{(\tilde{A})}$ that inspired the current work. In Section~\ref{SectionPresEllipticArtinGrps}, we introduce in Propositions~\ref{Prop_New_PresD4} and \ref{Prop_NewPres_Elliptic_En} the new presentations for the elliptic Artin groups. We also describe new elliptic diagrams to illustrate these presentations (see Figures~\ref{Fig_New_Diagram_D4} and \ref{Fig_New_Diagram_En}). Finally, after introducing several combinatorial techniques, Section~\ref{SectionCancelEllipticArtinMon} shows that the monoids defined by these presentations are cancellative (see Theorem~\ref{Thm_Cancellativity}).\\

\textit{Acknowledgements.} This work was supported by the German Research Foundation (SFB-TRR 358/1 2023 - 491392403).

\section{The affine Artin group of type $\tilde{A}$}\label{SectionAffineArtinGroupTildeA}

In this section, we will define three presentations of the affine Artin group of type $\tilde{A}_n$ for $n \geq 2$. The first, given in Definition~\ref{Def_StandardPres_TildeA}, is the classical presentation. The second one (Definition~\ref{Def_PresShi_TildeA}) is due to Shi~\cite{Shi} and the third (Definition~\ref{Def_PresCorranLeeLee_TildeA}) is due to Corran--Lee--Lee~\cite{CorranLeeLee}. The latter presentation was used in \cite{CorranLeeLee} to describe presentations of the complex braid groups associated to imprimitive complex reflection groups. Another notable application of Definition~\ref{Def_PresCorranLeeLee_TildeA} is the existence of interval Garside structures for the affine Artin group of type $\tilde{A}$ \cite{Neaime_TildeA} (see also \cite{NeaimeEEN}).

\begin{defn}[Classical presentation]\label{Def_StandardPres_TildeA}
    Let $\art(\tilde{A}_{n-1})$ be the affine Artin group of type $\tilde{A}_{n-1}$ for $n \geq 3$.
    The classical presentation is defined by generators $r_i$, $i \in \mathbb{Z}/n\mathbb{Z}$ and relations:
    \begin{itemize}
        \item[(1)] $r_ir_{i+1}r_i = r_{i+1}r_ir_{i+1}$ for $i \in \mathbb{Z}/n\mathbb{Z}$
        \item[(2)] $r_ir_j = r_jr_i$ for $i < j$, where $i,j \in \mathbb{Z}/n\mathbb{Z}$.
    \end{itemize}
    Adding the quadratic relations $r_i^2=1$, we obtain a presentation of the Coxeter group $\cox(\tilde{A}_{n-1})$ of type $\tilde{A}_{n-1}$. 
\end{defn}

The presentation given in Definition~\ref{Def_StandardPres_TildeA} is described by the diagram of Figure~\ref{Fig_StandardDiag_TildeA}. We use the standard conventions of Coxeter diagrams. Relation $(1)$ is described by an unlabelled edge. Relation $(2)$ describes the fact that there is no edge between the corresponding vertices.


\begin{figure}[H]
    \centering
    \begin{tikzpicture}[mynode/.style={circle, draw, fill=black!50, inner sep=0pt, minimum width=4pt}]
        \node[mynode,label=below:$r_1$] (r1) at (0,0) {};
        \node[mynode,label=below:$r_2$] (r2) at (2,0) {};
        \node[mynode,label=below:$r_3$] (r3) at (4,0) {};
        \node[mynode] (rnm2) at (6,0) {};
        \node[mynode,label=below:$r_{n-1}$] (rnm1) at (8,0) {};
        \node[mynode,label=above:$r_{n}$] (rn) at (4,1.5) {};

        \draw[-] (r1) to (r2);
        \draw[-] (r2) to (r3);
        \draw[dashed,-] (r3) to (rnm2);
        \draw[-] (rnm2) to (rnm1);
        \draw[dashed,-] (r1) to (rn);
        \draw[dashed,-] (rnm1) to (rn);
    \end{tikzpicture}
    \caption{The diagram of the presentation of $\art(\tilde{A}_{n-1})$ given in Definition \ref{Def_StandardPres_TildeA}.}
    \label{Fig_StandardDiag_TildeA}
\end{figure}
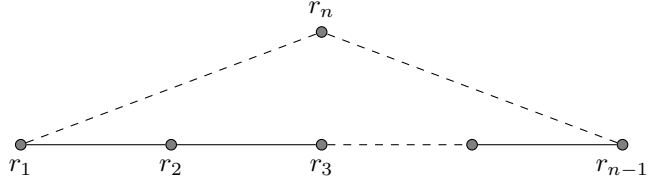

\begin{defn}[Presentation of Shi]\label{Def_PresShi_TildeA}
    The group $\art(\tilde{A}_{n-1})$ is isomorphic to the group defined by a presentation with generators $t_0$, $t_1$, $r_3$, $r_4$, $\cdots$, $r_{n}$ and relations:
    \begin{itemize}
        \item[(1)] $r_ir_{i+1}r_i = r_{i+1}r_ir_{i+1}$ for $3 \leq i \leq n-1$
        \item[(2)] $r_ir_j = r_jr_i$ for $3 \leq i < j \leq n$
        \item[(3)] $r_3 t_i r_3 = t_i r_3 t_i$ for $i=0,1$
        \item[(4)] $r_j t_i = t_i r_j$ for $i=0,1$ and $4 \leq j \leq n$
        \item[(5)] $r_3 t_1t_0 r_3 t_1t_0 = t_1t_0 r_3 t_1t_0 r_3$.
    \end{itemize}
    Adding the quadratic relations $t_0^2 = t_1^2 = r_i^2=1$, for $3 \leq i \leq n$, we obtain a presentation of the Coxeter group $\cox(\tilde{A}_{n-1})$ of type $\tilde{A}_{n-1}$. 
\end{defn}

The presentation described in Definition~\ref{Def_PresShi_TildeA} can be illustrated by the diagram of Figure~\ref{Fig_ShiDiagram_TildeA}. The label $\infty$ on the edge between $t_0$ and $t_1$ means that there is no relation between $t_0$ and $t_1$. The double edge inside the triangle formed by $t_0$, $t_1$, and $r_3$ illustrates the relation $(5)$ of the presentation.

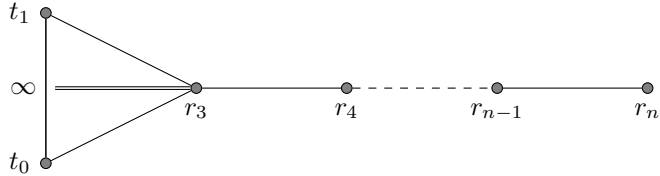
\begin{figure}[H]
    \centering
    \begin{tikzpicture}[mynode/.style={circle, draw, fill=black!50, inner sep=0pt, minimum width=4pt}]
        \node[mynode,label=left:$t_0$] (t0) at (0,0) {};
        \node[mynode,label=left:$t_1$] (t1) at (0,2) {};
        \node[mynode,label=below:$r_3$] (r3) at (2,1) {};
        \node[mynode,label=below:$r_4$] (r4) at (4,1) {};
        \node[mynode,label=below:$r_{n-1}$] (rnm1) at (6,1) {};
        \node[mynode,label=below:$r_n$] (rn) at (8,1) {};
        \node[] (0) at (0,1) {};
    
        \draw[-] (t0) to node[auto] {$\infty$} (t1);
        \draw[-] (t0) to (t1);
        \draw[-] (t0) to (r3);
        \draw[-] (t1) to (r3);
        \draw[-] (r3) to (r4);
        \draw[dashed,-] (r4) to (rnm1);
        \draw[-] (rnm1) to (rn);
        \draw[double,-] (r3) to (0);
    \end{tikzpicture}
    \caption{The diagram of the presentation of $\art(\tilde{A}_{n-1})$ given in Definition~\ref{Def_PresShi_TildeA}.}
    \label{Fig_ShiDiagram_TildeA}
\end{figure}

\begin{defn}[Presentation of Corran--Lee--Lee]\label{Def_PresCorranLeeLee_TildeA}
    The group $\art(\tilde{A}_{n-1})$ is isomorphic to the group defined by a presentation with generators $t_i$ ($i \in \mathbb{Z}$), $r_3$, $r_4$, $\cdots$, $r_n$ and relations:
    	\begin{itemize}
    	    \item[(1)] $r_ir_{i+1}r_i = r_{i+1}r_ir_{i+1}$ for $3 \leq i \leq n-1$
                \item[(2)] $r_ir_j = r_jr_i$ for $3 \leq i < j \leq n$
                \item[(3)] $r_3 t_i r_3 = t_i r_3 t_i$ for $i \in \mathbb{Z}$
                \item[(4)] $r_j t_i = t_i r_j$ for $i \in \mathbb{Z}$ and $4 \leq j \leq n$
                \item[(5)] $t_i t_{i-1} = t_j t_{j-1}$ for $i,j \in \mathbb{Z}$.
    	\end{itemize}
    Adding the quadratic relations $t_i^2 = 1$ for $i \in \mathbb{Z}$ and $r_i^2=1$, for $3 \leq i \leq n$, we obtain a presentation of the Coxeter group $\cox(\tilde{A}_{n-1})$ of type $\tilde{A}_{n-1}$.
\end{defn}

The presentation described in Definition~\ref{Def_PresCorranLeeLee_TildeA} can be illustrated by the following diagram. Notice that the line with labels $t_i$ for $i\in \mathbb{Z}$ illustrates the relation $(5)$ of the presentation.

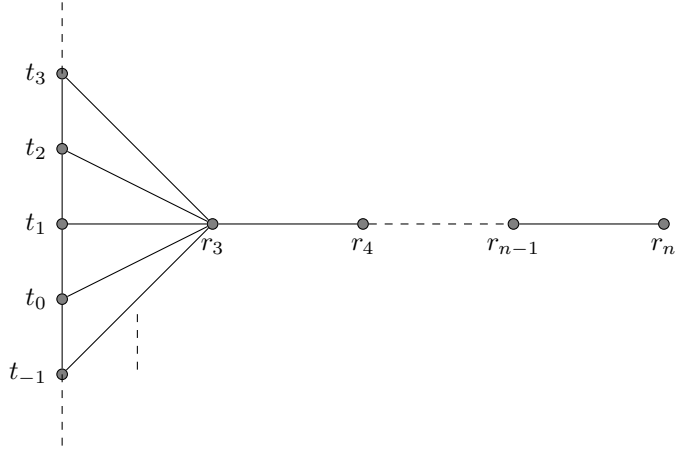
\begin{figure}[H]
    \centering
    \begin{tikzpicture}[mynode/.style={circle, draw, fill=black!50, inner sep=0pt, minimum width=4pt}]
        \node[mynode,label=left:$t_0$] (t0) at (0,0) {};
        \node[mynode,label=left:$t_{-1}$] (tm1) at (0,-1) {};
        \node[mynode,label=left:$t_1$] (t1) at (0,1) {};
        \node[mynode,label=left:$t_2$] (t2) at (0,2) {};
        \node[mynode,label=left:$t_3$] (t3) at (0,3) {};
        
        \node[mynode,label=below:$r_3$] (r3) at (2,1) {};
        \node[mynode,label=below:$r_4$] (r4) at (4,1) {};
        \node[mynode,label=below:$r_{n-1}$] (rnm1) at (6,1) {};
        \node[mynode,label=below:$r_n$] (rn) at (8,1) {};
        \node[] (0) at (0,1) {};

        \draw[-] (t0) to (t1);
        \draw[-] (t1) to (t2);
        \draw[-] (t2) to (t3);
        \draw[-] (t0) to (tm1);
        \draw[dashed,-] (0,-1) -- (0,-2);
        \draw[dashed,-] (0,3) -- (0,4);

        \draw[-] (t0) to (r3);
        \draw[-] (t1) to (r3);
        \draw[-] (t2) to (r3);
        \draw[-] (t3) to (r3);
        \draw[-] (tm1) to (r3);
        \draw[dashed,-] (1,-0.2) to (1,-1);
        
        \draw[-] (r3) to (r4);
        \draw[dashed,-] (r4) to (rnm1);
        \draw[-] (rnm1) to (rn);
    \end{tikzpicture}
    \caption{The diagram of the presentation of $\art(\tilde{A}_{n-1})$ given in Definition~\ref{Def_PresCorranLeeLee_TildeA}.}
    \label{Fig_CorranLeeLeeDiagram_TildeA}
\end{figure}

Note that in Definition~\ref{Def_PresCorranLeeLee_TildeA}, we are attaching the so-called dual presentation of the infinite dihedral group corresponding to the infinite generators $t_i$, $i \in \mathbb{Z}$ (this is represented by the line in Figure~\ref{Fig_CorranLeeLeeDiagram_TildeA}) to classical presentations of type $A$ corresponding to the finite subsets of generators $\{t_i, r_3, r_4, \cdots, r_n\}$ for each $i \in \mathbb{Z}$.

\begin{exmp}\label{Ex_TypeA2Tilde}
    We will illustrate the $3$ generating sets of $\cox(\tilde{A}_{2})$ given in Definitions~\ref{Def_StandardPres_TildeA}, \ref{Def_PresShi_TildeA}, and \ref{Def_PresCorranLeeLee_TildeA}. Note that we are analysing here the generators in the Coxeter group - and not in the Artin group -, so all the generators are reflections. Consider for that the tiling of $\mathbb{R}^2$ be equilateral triangles as shown in Figure~\ref{Fig_A2Tilde_Tiling}. This corresponds to the Coxeter complex of type $\tilde{A}_2$. The fundamental chamber is denoted by $\mathcal{C}$. The $3$ hyperplanes of $\mathcal{C}$ correspond to the standard generating set of reflections $r_1,r_2$, and $r_3$ as given in Definition~\ref{Def_StandardPres_TildeA}. The hyperplanes that correspond to $r_1, r_2$, and $r_3$ are coloured in thick green, red, and blue, respectively. In order to obtain the generating set of Definition~\ref{Def_PresShi_TildeA}, consider $t_1 := r_1r_3r_1$. Then the generators of $\cox(\tilde{A}_{2})$ are $r_2 := t_0$, $t_1$, and $r_3$. Notice that the hyperplanes of $t_0$ and $t_1$ are parallel, and the element $t_1t_0$ is a translation. Now, take all the reflections whose hyperplanes are parallel to those of $t_0$ and $t_1$. This corresponds to the reflections $t_i$, $i \in \mathbb{Z}$ of the infinite dihedral group. The related hyperplanes are coloured in red. Adding $r_3$ to these reflections, we obtain the generating set of Definition~\ref{Def_PresCorranLeeLee_TildeA}.
\end{exmp}

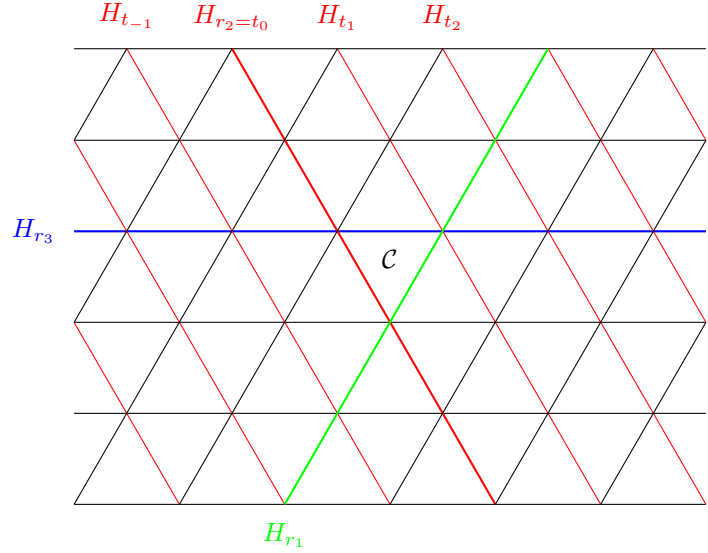
\begin{figure}[H]
    \centering
    \begin{tikzpicture}[scale=0.7]
        \draw[-] (-6,-1.73) -- (6,-1.73); 
        \draw[-] (-6,-3.46) -- (6,-3.46); 
        \draw[-] (-6,0) -- (6,0);
        \draw[-,thick,color=blue] (-6,1.73) -- (6,1.73);
        \node[label=left:\textcolor{blue}{$H_{r_3}$}] (Hs3) at (-6,1.73) {};
        \draw[-] (-6,3.46) -- (6,3.46);
        \draw[-] (-6,5.2) -- (6,5.2);

        \draw[-,color=red] (-6,0) -- (-4,-3.46);
        \draw[-,color=red] (-6,3.46) -- (-2,-3.46);
        \draw[-,color=red] (-5,5.2) -- (0,-3.46);
        \node[label=above:\textcolor{red}{$H_{t_{-1}}$}] (Htm1) at (-5,5.2) {};
        \draw[-,thick,color=red] (-3,5.2) -- (2,-3.46);
        \node[label=above:\textcolor{red}{$H_{r_2=t_0}$}] (Ht0) at (-3,5.2) {};
        \draw[-,color=red] (-1,5.2) -- (4,-3.46);
        \node[label=above:\textcolor{red}{$H_{t_1}$}] (Ht1) at (-1,5.2) {};
        \draw[-,color=red] (1,5.2) -- (6,-3.46);
        \node[label=above:\textcolor{red}{$H_{t_2}$}] (Ht2) at (1,5.2) {};
        \draw[-,color=red] (3,5.2) -- (6,0);
        \draw[-,color=red] (5,5.2) -- (6,3.46);

        \draw[-] (-5,5.2) -- (-6,3.46);
        \draw[-] (-3,5.2) -- (-6,0);
        \draw[-] (-1,5.2) -- (-6,-3.46);
        \draw[-] (1,5.2) -- (-4,-3.46);
        \draw[-,thick,color=green] (3,5.2) -- (-2,-3.46);
        \node[label=below:\color{green}{$H_{r_1}$}] (Hr1) at (-2,-3.46) {};
        \draw[-] (5,5.2) -- (0,-3.46);
        \draw[-] (6,3.46) -- (2,-3.46);
        \draw[-] (6,0) -- (4,-3.46);

        \node[] (C) at (0,1.2) {$\mathcal{C}$};
    \end{tikzpicture}
    \caption{A portion of the Coxeter complex of type $\tilde{A}_2$, where we show the $3$ generating sets of $\cox(\tilde{A}_{2})$ given in Definitions~\ref{Def_StandardPres_TildeA}, \ref{Def_PresShi_TildeA}, and \ref{Def_PresCorranLeeLee_TildeA}.}
    \label{Fig_A2Tilde_Tiling}
\end{figure}

\section{Presentations of elliptic Artin groups}\label{SectionPresEllipticArtinGrps}

In this section, we define new presentations for the elliptic Artin groups of the simply laced types, namely types $\Gamma = \tilde{D}_4^{(1,1)}$ and $\Gamma = \tilde{E}_n^{(1,1)}$ for $n=6,7,8$, which is Saito's notation~\cite{SaitoI}. This is done in Propositions~\ref{Prop_New_PresD4} and \ref{Prop_NewPres_Elliptic_En}. We recall the usual presentations for these types in Definitions~\ref{Def_PresD4_Yamada} and \ref{DefPres_Elliptic_En_Yamada}. Our new presentations are described in diagram presentations given in Figures~\ref{Fig_New_Diagram_D4} and \ref{Fig_New_Diagram_En}. Figures~\ref{Fig_Diagram_Saito_D4} and \ref{Fig_Elliptic_E_6,7,8} describe the usual elliptic Dynkin diagrams of types $\tilde{D}_4^{(1,1)}$ and $\tilde{E}_n^{(1,1)}$ for $n=6,7,8$.

\begin{defn}[Presentation for type $\tilde{D}_4^{(1,1)}$]\label{Def_PresD4_Yamada}
    The elliptic Artin group $\art{(\tilde{D}_4^{(1,1)})}$ of type $\tilde{D}_4^{(1,1)}$ is isomorphic to the group $G(\tilde{D}_4^{(1,1)})$ defined by a presentation with generators $t_0,t_1,s_1,s_2,s_3,s_4$ and relations:
    \begin{itemize}
        \item[$(P_1)$:] $t_is_jt_i = s_jt_is_j$ for $i=0,1$ and $j=1,2,3,4$.
        \item[$(P_2)$:] $s_is_j = s_js_i$ for $i \neq j \in \{1,2,3,4\}$.
        \item[$(P_3)$:] $s_it_1t_0s_it_1t_0 = t_1t_0s_it_1t_0s_i$ for $i = 1,2,3,4$.
    \end{itemize}
    Adding the quadratic relations $t_i^2=1$ for $i=0,1$ and $s_j^2=1$ for $j=1,2,3,4$, we obtain a presentation of the group $\hyp{(\tilde{D}_4^{(1,1)})}$.
\end{defn}

The presentation described in Definition~\ref{Def_PresD4_Yamada} can be illustrated by the elliptic Dynkin diagram due to Saito \cite{SaitoI} given in Figure~\ref{Fig_Diagram_Saito_D4}.

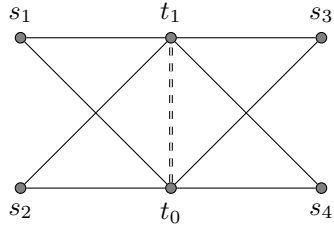
\begin{figure}[H]
    \centering
    \begin{tikzpicture}[mynode/.style={circle, draw, fill=black!50, inner sep=0pt, minimum width=4pt}]
        \node[mynode,label=below:$t_0$] (t0) at (0,0) {};
        \node[mynode,label=below:$s_2$] (s2) at (-2,0) {};
        \node[mynode,label=below:$s_4$] (s4) at (2,0) {};
        \node[mynode,label=above:$s_1$] (s1) at (-2,2) {};
        \node[mynode,label=above:$t_1$] (t1) at (0,2) {};
        \node[mynode,label=above:$s_3$] (s3) at (2,2) {};

        \draw[-] (t0) to (s1);
        \draw[-] (t0) to (s2);
        \draw[-] (t0) to (s3);
        \draw[-] (t0) to (s4);
        \draw[-] (t1) to (s1);
        \draw[-] (t1) to (s2);
        \draw[-] (t1) to (s3);
        \draw[-] (t1) to (s4);
        \draw[dashed, double, -] (t0) to (t1);
    \end{tikzpicture}
    \caption{Elliptic Dynkin diagram of type $\tilde{D}_4^{(1,1)}$.}
    \label{Fig_Diagram_Saito_D4}
\end{figure}


%
%

We will describe in Proposition~\ref{Prop_New_PresD4} a new presentation of the elliptic Arin group $\art{(\tilde{D}_4^{(1,1)})}$ and of $\hyp(\tilde{D}_4^{(1,1)})$. If we look at the portion of the presentation in Definition~\ref{Def_PresD4_Yamada} that corresponds to the generators $t_0,t_1$, and one of the $s_i$'s, we obtain the same relations of type $\tilde{A}_2$ as given by Shi in Definition~\ref{Def_PresShi_TildeA}. Analogously with the presentation given in Definition~\ref{Def_PresCorranLeeLee_TildeA}, the idea here is again to define a new presentation by using a bigger (even infinite) generating set obtained by adding elements $t_i$ with $i\in \mathbb{Z}$ that correspond to all the reflections in $\hyp(\tilde{D}_4^{(1,1)})$ whose hyperplanes are parallel to $t_0$ and $t_1$. These reflections generate the infinite dihedral group inside $\hyp(\tilde{D}_4^{(1,1)})$.

\begin{prop}[New presentation for type $\tilde{D}_4^{(1,1)}$]\label{Prop_New_PresD4}
    The elliptic Artin group $\art{(\tilde{D}_4^{(1,1)})}$ of type $\tilde{D}_4^{(1,1)}$ is isomorphic to the group $G'(\tilde{D}_4^{(1,1)})$ defined by a presentation with generators $t_i$ for $i\in \mathbb{Z}$ ,$s_1,s_2,s_3,s_4$ and relations:
    \begin{itemize}
        \item[$(R_1)$:] $t_is_jt_i = s_jt_is_j$ for $i \in \mathbb{Z}$ and $j=1,2,3,4$.
        \item[$(R_2)$:] $s_is_j = s_js_i$ for $i \neq j \in \{1,2,3,4\}$.
        \item[$(R_3)$:] $t_it_{i-1} = t_jt_{j-1}$ for $i,j \in \mathbb{Z}$.
    \end{itemize}
    Adding the quadratic relations $t_i^2=1$ ($i \in \mathbb{Z}$) and $s_j^2=1$ for $j=1,2,3,4$, we obtain a presentation of the group $\hyp{(\tilde{D}_4^{(1,1)})}$.
\end{prop}

The presentation obtained in Proposition~\ref{Prop_New_PresD4} can be illustrated in the diagram presentation of 
Figure~\ref{Fig_New_Diagram_D4}. Relations $(R_1)$ and $(R_2)$ are illustrated in a standard way. Relation $(R_3)$ is illustrated by the middle vertical line with nodes $t_i, i \in \mathbb{Z}$. 

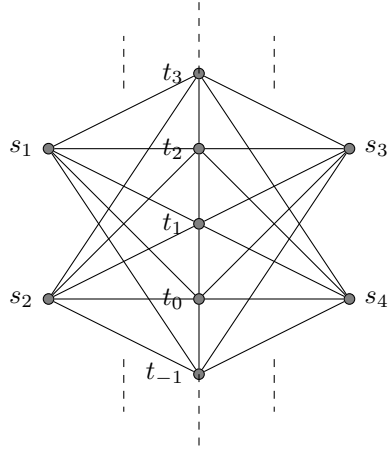
\begin{figure}[H]
    \centering
    \begin{tikzpicture}[mynode/.style={circle, draw, fill=black!50, inner sep=0pt, minimum width=4pt}]
        \node[mynode,label=left:$t_0$] (t0) at (0,0) {};
        \node[mynode,label=left:$t_1$] (t1) at (0,1) {};
        \node[mynode,label=left:$t_2$] (t2) at (0,2) {};
        \node[mynode,label=left:$t_3$] (t3) at (0,3) {};
        \node[mynode,label=left:$t_{-1}$] (tm1) at (0,-1) {};
        \node[mynode,label=left:$s_2$] (s2) at (-2,0) {};
        \node[mynode,label=right:$s_4$] (s4) at (2,0) {};
        \node[mynode,label=left:$s_1$] (s1) at (-2,2) {};
        \node[mynode,label=right:$s_3$] (s3) at (2,2) {};

        \draw[dashed,-] (0,-1) -- (0,-2);
        \draw[dashed,-] (0,3) -- (0,4);
        \draw[-] (t0) to (t1);
        \draw[-] (t1) to (t2);
        \draw[-] (t2) to (t3);
        \draw[-] (t0) to (tm1);

        \draw[-] (s1) to (t0);
        \draw[-] (s1) to (t1);
        \draw[-] (s1) to (t2);
        \draw[-] (s1) to (t3);
        \draw[-] (s1) to (tm1);
        
        \draw[-] (s2) to (t0);
        \draw[-] (s2) to (t1);
        \draw[-] (s2) to (t2);
        \draw[-] (s2) to (t3);
        \draw[-] (s2) to (tm1);

        \draw[-] (s3) to (t0);
        \draw[-] (s3) to (t1);
        \draw[-] (s3) to (t2);
        \draw[-] (s3) to (t3);
        \draw[-] (s3) to (tm1);

        \draw[-] (s4) to (t0);
        \draw[-] (s4) to (t1);
        \draw[-] (s4) to (t2);
        \draw[-] (s4) to (t3);
        \draw[-] (s4) to (tm1);

        \draw[dashed,-] (-1,-0.8) to (-1,-1.5);
        \draw[dashed,-] (1,-0.8) to (1,-1.5);
        \draw[dashed,-] (-1,2.8) to (-1,3.5);
        \draw[dashed,-] (1,2.8) to (1,3.5);
        
    \end{tikzpicture}
    \caption{New diagram presentation for type $\tilde{D}_4^{(1,1)}$.}
    \label{Fig_New_Diagram_D4}
\end{figure}

In order to prove Proposition~\ref{Prop_New_PresD4}, we need the following lemmas. Our proof is similar to the proof of Corran--Picantin \cite{CorranPicantin} in the case of complex reflection groups of type $(e,e,n)$.

\begin{lem}\label{LemP3InG'}
    In $G'(\tilde{D}_4^{(1,1)})$, we have that $$t_tt_0s_it_1t_0s_i = s_it_1t_0s_it_1t_0,$$
    for $i \in \{1,2,3,4\}.$
\end{lem}

\begin{proof}
    The equation of the lemma is given by the next calculation:
    \begin{center}
    \begin{tabular}{lll}
       $\underline{t_1t_0}s_it_1t_0s_i$  & = & $t_2\underline{t_1s_it_1}t_0s_i$ by using $(R_3)$ \\
         & $=$ & $t_2s_it_1\underline{s_it_0s_i}$ by using $(R_1)$\\
         & $=$ & $t_2s_i\underline{t_1t_0}s_it_0$ by using $(R_1)$\\
         & $=$ & $\underline{t_2s_it_2}t_1s_it_0$ by using $(R_3)$\\
         & $=$ & $s_it_2\underline{s_it_1s_i}t_0$ by using $(R_1)$\\
         & $=$ & $s_i\underline{t_2t_1}s_it_1t_0$ by using $(R_1)$\\
         & $=$ & $s_it_1t_0s_it_1t_0$.
    \end{tabular}
    \end{center}
\end{proof}

\begin{lem}\label{LemT_kInTermsT_0T_1}
    Let $i \in \mathbb{Z}$. In $G'(\tilde{D}_4^{(1,1)})$, we have that $$t_i = \left\{ \begin{array}{ll}
     \underset{i-1}{\underbrace{t_1t_0 \cdots t_0}}\ t_1\ \underset{i-1}{\underbrace{t_0^{-1}t_1^{-1} \cdots t_1^{-1}}}  & if\ i\ is\ odd, \\
      & \\
     \underset{i-1}{\underbrace{t_1t_0 \cdots t_1}}\ t_0\ \underset{i-1}{\underbrace{t_1^{-1}t_0^{-1} \cdots t_1^{-1}}}  & if\ i\ is\ even.
    \end{array} \right.
$$
\end{lem}

\begin{proof}
    Using Relation $(R_3)$, it is straightforward to compute any $t_k$ in terms of $t_0$, $t_1$, and their inverses as given in the lemma. 
\end{proof}

\begin{lem}\label{LemR3InG}
    In $G(\tilde{D}_4^{(1,1)})$, let $t_i$ be defined as in Lemma~\ref{LemT_kInTermsT_0T_1} for $i \in \mathbb{Z}$. We have that $$t_it_{i-1} = t_1t_0.$$
\end{lem}

\begin{proof}
    Let $i$ odd. From Lemma~\ref{LemT_kInTermsT_0T_1}, we have that $$t_i = \underset{i-1}{\underbrace{t_1t_0 \cdots t_0}}\ t_1\ \underset{i-1}{\underbrace{t_0^{-1}t_1^{-1} \cdots t_1^{-1}}}$$ and $$t_{i-1} = \underset{i-2}{\underbrace{t_1t_0 \cdots t_1}}\ t_0\ \underset{i-2}{\underbrace{t_1^{-1}t_0^{-1} \cdots t_1^{-1}}}.$$
    Multiplying these two expressions, we get $t_it_{i-1} = t_1t_0$. The case $i$ even is similar.
\end{proof}

\begin{lem}\label{LemR1InG}
    In $G(\tilde{D}_4^{(1,1)})$, let $t_i$ be defined as in Lemma~\ref{LemT_kInTermsT_0T_1} for $i \in \mathbb{Z}$. We have that $$s_3t_is_3 = t_is_3t_i,$$ for any $i \in \mathbb{Z}$.
\end{lem}

\begin{proof}
    By Lemma~\ref{LemR3InG}, we have that $t_it_{i-1} = t_jt_{j-1} =t_1t_0$ for all $i,j \in \mathbb{Z}$, then the proof of the lemma is the same as Lemma~2.9(a) in \cite{CorranPicantin}.
\end{proof}

\begin{proof}[Proof of Proposition~\ref{Prop_New_PresD4}.]

We first show that the natural map from $G(\tilde{D}_4^{(1,1)})$ to $G'(\tilde{D}_4^{(1,1)})$ that maps $t_i \mapsto t_i$ for $i=0,1$ and $s_j \mapsto s_j$ for $j\in\{1,2,3,4\}$ extends to a group homomorphism.

Actually, Relation $(P_1)$ is a particular case of Relation $(R_1)$. Relation $(P_2)$ is the same as $(R_3)$. By Lemma~\ref{LemP3InG'}, Relation $(P_3)$ holds in $G'(\tilde{D}_4^{(1,1)})$.

Next, define the map from $G'(\tilde{D}_4^{(1,1)})$ to $G(\tilde{D}_4^{(1,1)})$ that maps each $t_i$ to its expression over $\{t_0,t_1,t_0^{-1},t_1^{-1}\}$ given in Lemma~\ref{LemT_kInTermsT_0T_1}, and maps $s_j \mapsto s_j$ for $j\in\{1,2,3,4\}$. We show that this map extends to a group homomorphism.

Actually, by Lemma~\ref{LemR1InG}, Relation $(R_1)$ holds in $G$ for all $i \in \mathbb{Z}$. Relation $(R_2)$ is the same as $(P_2)$. By Lemma~\ref{LemR3InG}, Relation $(R_3)$ holds in $G$.

This shows that the groups $G(\tilde{D}_4^{(1,1)})$ and $G'(\tilde{D}_4^{(1,1)})$ are isomorphic.

The generating set of the new presentation of $\hyp{(\tilde{D}_4^{(1,1)})}$ consists of the same generators as the one in Definition~\ref{Def_PresD4_Yamada}, as well as some conjugates (see Lemma~\ref{LemT_kInTermsT_0T_1}). Since it is the case for the presentation in Definition~\ref{Def_PresD4_Yamada}, adding the quadratic relations to the new presentation gives a presentation for the group $\hyp{(\tilde{D}_4^{(1,1)})}$.
\end{proof}

We extend our result to all elliptic Artin groups of types $\tilde{E}_n^{(1,1)}$ for $n=6,7,8$. This covers all the simply laced cases of elliptic Weyl groups.

\begin{defn}[Presentations for types $\tilde{E}_n^{(1,1)}$]\label{DefPres_Elliptic_En_Yamada}
    The elliptic Artin group $\art(\tilde{E}_n^{(1,1)})$ for $n=6,7,8$ is isomorphic to the group $G(\tilde{E}_n^{(1,1)})$ defined by a presentation with generators $t_0, t_1, s_1, s_2, \cdots, s_n$, and the relations can be illustrated by the elliptic Dynkin diagrams shown in Figure~\ref{Fig_Elliptic_E_6,7,8}. 
    The relations between the generators $t_0,t_1,s_1,s_2$, and $s_3$ are similar to those of Definition~\ref{Def_PresD4_Yamada}. The other relations are standard Coxeter relations.

    Adding the quadratic relations for all the generators, we obtain a presentation of the group $\hyp(\tilde{E}_n^{(1,1)})$ for $n=6,7,8$.
\end{defn}

\begin{figure}[H]
        \begin{tikzpicture}[mynode/.style={circle, draw, fill=black!50, inner sep=0pt, minimum width=4pt},scale=0.8]
        \node[mynode,label=below:$t_0$] (t0) at (0,0) {};
        \node[mynode,label=below:$s_3$] (s3) at (2,0) {};
        \node[mynode,label=above:$s_1$] (s1) at (-2,1) {};
        \node[mynode,label=above:$t_1$] (t1) at (0,2) {};
        \node[mynode,label=above:$s_2$] (s2) at (2,2) {};

        \node[mynode, label=above:$s_4$] (s4) at (-4,1) {};
        \node[mynode, label=above:$s_5$] (s5) at (4,2) {};
        \node[mynode, label=below:$s_6$] (s6) at (4,0) {};

        \node[] (n6) at (-9,1) {\underline{Elliptic diagram $\tilde{E}_6^{(1,1)}$}:};

        \draw[-] (t0) to (s1);
        \draw[-] (t0) to (s2);
        \draw[-] (t0) to (s3);
        \draw[-] (t1) to (s1);
        \draw[-] (t1) to (s2);
        \draw[-] (t1) to (s3);
        \draw[dashed, double, -] (t0) to (t1);

        \draw[-] (s1) to (s4);
        \draw[-] (s2) to (s5);
        \draw[-] (s3) to (s6);
        \end{tikzpicture}

        \begin{tikzpicture}[mynode/.style={circle, draw, fill=black!50, inner sep=0pt, minimum width=4pt},scale=0.7]
        \node[mynode,label=below:$t_0$] (t0) at (0,0) {};
        \node[mynode,label=below:$s_3$] (s3) at (2,0) {};
        \node[mynode,label=above:$s_1$] (s1) at (-2,1) {};
        \node[mynode,label=above:$t_1$] (t1) at (0,2) {};
        \node[mynode,label=above:$s_2$] (s2) at (2,2) {};

        \node[mynode,label=above:$s_4$] (s4) at (4,2) {};
        \node[mynode,label=above:$s_5$] (s5) at (6,2) {};
        \node[mynode,label=below:$s_6$] (s6) at (4,0) {};
        \node[mynode,label=below:$s_7$] (s7) at (6,0) {};

        \node[] (n7) at (-9,1) {\underline{Elliptic diagram $\tilde{E}_7^{(1,1)}$}:};

        \draw[-] (t0) to (s1);
        \draw[-] (t0) to (s2);
        \draw[-] (t0) to (s3);
        \draw[-] (t1) to (s1);
        \draw[-] (t1) to (s2);
        \draw[-] (t1) to (s3);
        \draw[dashed, double, -] (t0) to (t1);

        \draw[-] (s2) to (s4);
        \draw[-] (s4) to (s5);
        \draw[-] (s3) to (s6);
        \draw[-] (s6) to (s7);
        \end{tikzpicture}

        \begin{tikzpicture}[mynode/.style={circle, draw, fill=black!50, inner sep=0pt, minimum width=4pt},scale=0.7]
        \node[mynode,label=below:$t_0$] (t0) at (0,0) {};
        \node[mynode,label=below:$s_3$] (s3) at (2,0) {};
        \node[mynode,label=above:$s_1$] (s1) at (-2,1) {};
        \node[mynode,label=above:$t_1$] (t1) at (0,2) {};
        \node[mynode,label=above:$s_2$] (s2) at (2,2) {};

        \node[mynode,label=above:$s_4$] (s4) at (4,2) {};
        \node[mynode,label=below:$s_5$] (s5) at (4,0) {};
        \node[mynode,label=below:$s_6$] (s6) at (6,0) {};
        \node[mynode,label=below:$s_7$] (s7) at (8,0) {};
        \node[mynode,label=below:$s_8$] (s8) at (10,0) {};

        \node[] (n8) at (-6,1) {\underline{Elliptic diagram $\tilde{E}_8^{(1,1)}$}:};

        \draw[-] (t0) to (s1);
        \draw[-] (t0) to (s2);
        \draw[-] (t0) to (s3);
        \draw[-] (t1) to (s1);
        \draw[-] (t1) to (s2);
        \draw[-] (t1) to (s3);
        \draw[dashed, double, -] (t0) to (t1);

        \draw[-] (s2) to (s4);
        \draw[-] (s3) to (s5);
        \draw[-] (s5) to (s6);
        \draw[-] (s6) to (s7);
        \draw[-] (s7) to (s8);
        \end{tikzpicture}
    \caption{Elliptic Dynkin diagrams of types $\tilde{E}_n^{(1,1)}$ for $n=6,7,8$.}
    \label{Fig_Elliptic_E_6,7,8}
\end{figure}
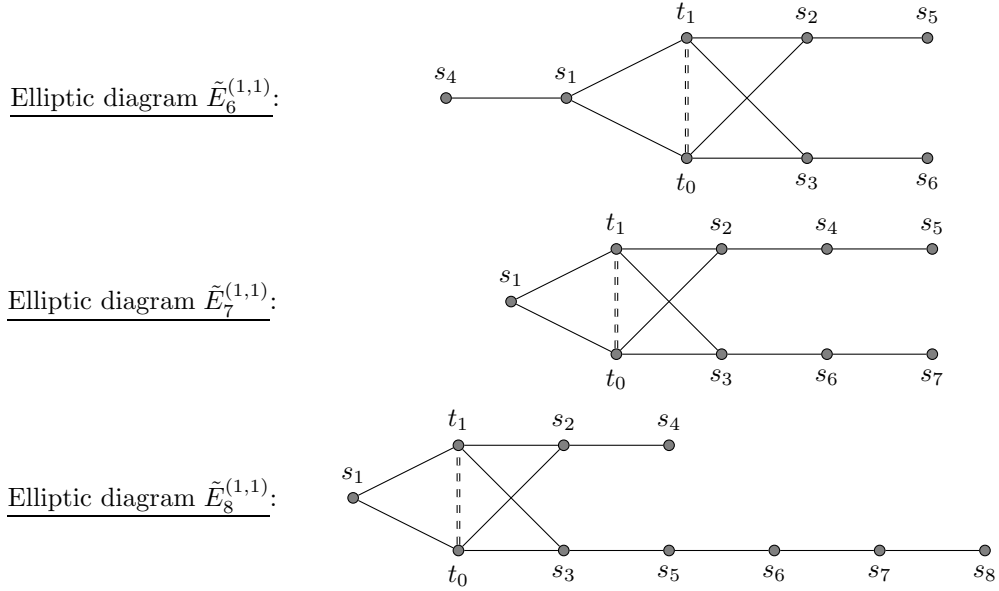

\begin{prop}[New presentations for types $\tilde{E}_n^{(1,1)}$]\label{Prop_NewPres_Elliptic_En}
    The elliptic Artin group $\art{(\tilde{E}_n^{(1,1)})}$ of type $\tilde{E}_n^{(1,1)}$ for $n=6,7,8$ is isomorphic to the group $G'(\tilde{E}_n^{(1,1)})$ defined by a presentation with generators $t_i$ for $i\in \mathbb{Z}$ and $s_1,s_2, \cdots, s_n$ and the  relations can be illustrated by the diagrams given in Figure~\ref{Fig_New_Diagram_En}. The relations between the generators $t_i$ ($i \in \mathbb{Z}$), $s_1,s_2$, and $s_3$ are similar to those given in Proposition~\ref{Prop_New_PresD4}. The other relations are standard Coxeter relations.
    
    Adding the quadratic relations for all the generators, we obtain a presentation of the group $\hyp{(\tilde{E}_n^{(1,1)})}$ for $n=6,7,8$.
\end{prop} 

\begin{figure}[H]
    \centering
    \begin{tikzpicture}[mynode/.style={circle, draw, fill=black!50, inner sep=0pt, minimum width=4pt},scale=0.9]
        \node[mynode,label=left:$t_0$] (t0) at (0,0) {};
        \node[mynode,label=left:$t_1$] (t1) at (0,1) {};
        \node[mynode,label=left:$t_2$] (t2) at (0,2) {};
        \node[mynode,label=left:$t_3$] (t3) at (0,3) {};
        \node[mynode,label=left:$t_{-1}$] (tm1) at (0,-1) {};
        \node[mynode,label=below:$s_3$] (s3) at (2,0) {};
        \node[mynode,label=above:$s_1$] (s1) at (-2,1) {};
        \node[mynode,label=above:$s_2$] (s2) at (2,2) {};

        \node[mynode,label=above:$s_4$] (s4) at (-4,1) {};
        \node[mynode,label=above:$s_5$] (s5) at (4,2) {};
        \node[mynode,label=below:$s_6$] (s6) at (4,0) {};

        \node[] (n6) at (-9,1) {\underline{Type $\tilde{E}_6^{(1,1)}$}:};

        \draw[dashed,-] (0,-1) -- (0,-2);
        \draw[dashed,-] (0,3) -- (0,4);
        \draw[-] (t0) to (t1);
        \draw[-] (t1) to (t2);
        \draw[-] (t2) to (t3);
        \draw[-] (t0) to (tm1);

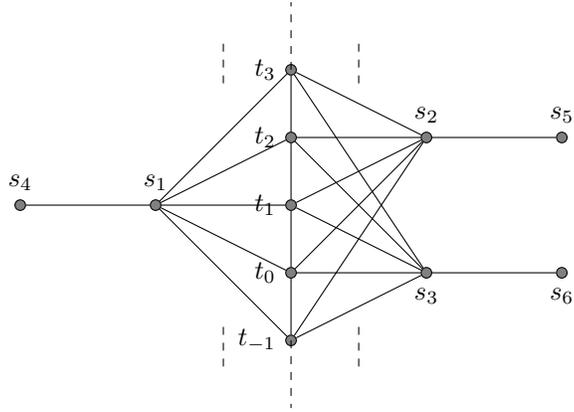
        \draw[-] (s1) to (t0);
        \draw[-] (s1) to (t1);
        \draw[-] (s1) to (t2);
        \draw[-] (s1) to (t3);
        \draw[-] (s1) to (tm1);

        \draw[-] (s3) to (t0);
        \draw[-] (s3) to (t1);
        \draw[-] (s3) to (t2);
        \draw[-] (s3) to (t3);
        \draw[-] (s3) to (tm1);

        \draw[-] (s2) to (t0);
        \draw[-] (s2) to (t1);
        \draw[-] (s2) to (t2);
        \draw[-] (s2) to (t3);
        \draw[-] (s2) to (tm1);

        \draw[-] (s1) to (s4);
        \draw[-] (s2) to (s5);
        \draw[-] (s3) to (s6);

        \draw[dashed,-] (-1,-0.8) to (-1,-1.5);
        \draw[dashed,-] (1,-0.8) to (1,-1.5);
        \draw[dashed,-] (-1,2.8) to (-1,3.5);
        \draw[dashed,-] (1,2.8) to (1,3.5);
    \end{tikzpicture}

    \vspace{0.5cm}

    \begin{tikzpicture}[mynode/.style={circle, draw, fill=black!50, inner sep=0pt, minimum width=4pt},scale=0.9]
        \node[mynode,label=left:$t_0$] (t0) at (0,0) {};
        \node[mynode,label=left:$t_1$] (t1) at (0,1) {};
        \node[mynode,label=left:$t_2$] (t2) at (0,2) {};
        \node[mynode,label=left:$t_3$] (t3) at (0,3) {};
        \node[mynode,label=left:$t_{-1}$] (tm1) at (0,-1) {};
        \node[mynode,label=below:$s_3$] (s3) at (2,0) {};
        \node[mynode,label=above:$s_1$] (s1) at (-2,1) {};
        \node[mynode,label=above:$s_2$] (s2) at (2,2) {};

        \node[mynode,label=above:$s_4$] (s4) at (4,2) {};
        \node[mynode,label=above:$s_5$] (s5) at (6,2) {};
        \node[mynode,label=below:$s_6$] (s6) at (4,0) {};
        \node[mynode,label=below:$s_7$] (s7) at (6,0) {};

        \node[] (n7) at (-7,1) {\underline{Type $\tilde{E}_7^{(1,1)}$}:};

        \draw[dashed,-] (0,-1) -- (0,-2);
        \draw[dashed,-] (0,3) -- (0,4);
        \draw[-] (t0) to (t1);
        \draw[-] (t1) to (t2);
        \draw[-] (t2) to (t3);
        \draw[-] (t0) to (tm1);

        \draw[-] (s1) to (t0);
        \draw[-] (s1) to (t1);
        \draw[-] (s1) to (t2);
        \draw[-] (s1) to (t3);
        \draw[-] (s1) to (tm1);

        \draw[-] (s3) to (t0);
        \draw[-] (s3) to (t1);
        \draw[-] (s3) to (t2);
        \draw[-] (s3) to (t3);
        \draw[-] (s3) to (tm1);

        \draw[-] (s2) to (t0);
        \draw[-] (s2) to (t1);
        \draw[-] (s2) to (t2);
        \draw[-] (s2) to (t3);
        \draw[-] (s2) to (tm1);

        \draw[-] (s2) to (s4);
        \draw[-] (s4) to (s5);
        \draw[-] (s3) to (s6);
        \draw[-] (s7) to (s6);

        \draw[dashed,-] (-1,-0.8) to (-1,-1.5);
        \draw[dashed,-] (1,-0.8) to (1,-1.5);
        \draw[dashed,-] (-1,2.8) to (-1,3.5);
        \draw[dashed,-] (1,2.8) to (1,3.5);
    \end{tikzpicture}

    \vspace{0.5cm}

    \begin{tikzpicture}[mynode/.style={circle, draw, fill=black!50, inner sep=0pt, minimum width=4pt},scale=0.85]
        \node[mynode,label=left:$t_0$] (t0) at (0,0) {};
        \node[mynode,label=left:$t_1$] (t1) at (0,1) {};
        \node[mynode,label=left:$t_2$] (t2) at (0,2) {};
        \node[mynode,label=left:$t_3$] (t3) at (0,3) {};
        \node[mynode,label=left:$t_{-1}$] (tm1) at (0,-1) {};
        \node[mynode,label=below:$s_3$] (s3) at (2,0) {};
        \node[mynode,label=above:$s_1$] (s1) at (-2,1) {};
        \node[mynode,label=above:$s_2$] (s2) at (2,2) {};

        \node[mynode,label=above:$s_4$] (s4) at (4,2) {};
        \node[mynode,label=below:$s_5$] (s5) at (4,0) {};
        \node[mynode,label=below:$s_6$] (s6) at (6,0) {};
        \node[mynode,label=below:$s_7$] (s7) at (8,0) {};
        \node[mynode,label=below:$s_8$] (s8) at (10,0) {};

        \node[] (n8) at (-4,1) {\underline{Type $\tilde{E}_8^{(1,1)}$}:};

        \draw[dashed,-] (0,-1) -- (0,-2);
        \draw[dashed,-] (0,3) -- (0,4);
        \draw[-] (t0) to (t1);
        \draw[-] (t1) to (t2);
        \draw[-] (t2) to (t3);
        \draw[-] (t0) to (tm1);

        \draw[-] (s1) to (t0);
        \draw[-] (s1) to (t1);
        \draw[-] (s1) to (t2);
        \draw[-] (s1) to (t3);
        \draw[-] (s1) to (tm1);

        \draw[-] (s3) to (t0);
        \draw[-] (s3) to (t1);
        \draw[-] (s3) to (t2);
        \draw[-] (s3) to (t3);
        \draw[-] (s3) to (tm1);

        \draw[-] (s2) to (t0);
        \draw[-] (s2) to (t1);
        \draw[-] (s2) to (t2);
        \draw[-] (s2) to (t3);
        \draw[-] (s2) to (tm1);

        \draw[-] (s2) to (s4);
        \draw[-] (s3) to (s5);
        \draw[-] (s5) to (s6);
        \draw[-] (s6) to (s7);
        \draw[-] (s7) to (s8);

        \draw[dashed,-] (-1,-0.8) to (-1,-1.5);
        \draw[dashed,-] (1,-0.8) to (1,-1.5);
        \draw[dashed,-] (-1,2.8) to (-1,3.5);
        \draw[dashed,-] (1,2.8) to (1,3.5);
    \end{tikzpicture}
    \caption{New diagram presentations for types $\tilde{E}_n^{(1,1)}$ for $n=6,7,8$.}
    \label{Fig_New_Diagram_En}
\end{figure}

\begin{proof}[Proof of Proposition~\ref{Prop_NewPres_Elliptic_En}]
    The proof is similar to the one of Proposition~\ref{Prop_New_PresD4}.

    Define the natural map from $G(\tilde{E}_n^{(1,1)})$ to $G'(\tilde{E}_n^{(1,1)})$ that maps $t_i \mapsto t_i$ for $i=0,1$ and $s_j \mapsto s_j$ for $1 \leq j \leq n$. 
    For all the relations involving $t_0,t_1,s_1,s_2$, and $s_3$, the proof is the same as for Proposition~\ref{Prop_New_PresD4}. For all the other relations between $s_4, s_5, \cdots, s_n$, they are identical in $G(\tilde{E}_n^{(1,1)})$ and $G'(\tilde{E}_n^{(1,1)})$. 
    
    Define now the map from $G'(\tilde{E}_n^{(1,1)})$ to $G(\tilde{E}_n^{(1,1)})$ that maps each $t_i$ with $i \in \mathbb{Z}$ to its expression over $\{t_0,t_1,t_0^{-1},t_1^{-1}\}$ as given in Lemma~\ref{LemT_kInTermsT_0T_1}, and maps $s_j \mapsto s_j$ for $1 \leq j \leq n$.
    Again, for all the relations involving $t_0,t_1,s_1,s_2$, and $s_3$, the proof is the same as for Proposition~\ref{Prop_New_PresD4}. Also, all the other relations between $s_4, s_5, \cdots, s_n$ are identical in $G(\tilde{E}_n^{(1,1)})$ and $G'(\tilde{E}_n^{(1,1)})$. We only still need to check that in $G(\tilde{E}_n^{(1,1)})$, each $s_j$ for $4 \leq j \leq n$ commutes with each of the expressions of $t_k$ given in Lemma~\ref{LemT_kInTermsT_0T_1}, but this is straightforward since $s_j$ for $4 \leq j \leq n$ commutes with $t_0$ and $t_1$. 
    
    This shows that the groups $G(\tilde{E}_n^{(1,1)})$ and $G'(\tilde{E}_n^{(1,1)})$ are isomorphic.

    The proof of the last statement of the proposition is the same as the last statement of Proposition~\ref{Prop_New_PresD4}.
\end{proof}  

\begin{rem}
    Adding the relation that describes the order of the Coxeter transformation in $W(\Gamma)$ to the presentation of $\hyp{(\Gamma)}$ in Propositions~\ref{Prop_New_PresD4} and \ref{Prop_NewPres_Elliptic_En}, we obtain a new presentation of the elliptic Weyl group $W(\Gamma)$. This is straightforward from the central extension mentioned in the second paragraph of Section~\ref{SectionIntroduction}.
\end{rem}

\section{Cancellativity of elliptic Artin monoids}\label{SectionCancelEllipticArtinMon}

Recall that a monoid $M$ is said to be cancellative if it satisfies the property:
\begin{center} 
$u v w = u v' w$ for $u,v,w \in M$ implies $v = v'$ in $M$.
\end{center}

In this section, we are considering the monoids defined by the same positive presentations as $G(\tilde{D}_4^{(1,1)})$, $G(\tilde{E}_n^{(1,1)})$, $G'(\tilde{D}_4^{(1,1)})$, $G'(\tilde{E}_n^{(1,1)})$, for $n=6,7,8$ of Definitions~\ref{Def_PresD4_Yamada}, \ref{DefPres_Elliptic_En_Yamada}, and Propositions~\ref{Prop_New_PresD4}, \ref{Prop_NewPres_Elliptic_En}, respectively. These monoids will be denoted by $M(\tilde{D}_4^{(1,1)})$, $M(\tilde{E}_n^{(1,1)})$, $M'(\tilde{D}_4^{(1,1)})$, and $M'(\tilde{E}_n^{(1,1)})$, respectively.

Consider the monoids $M(\tilde{D}_4^{(1,1)})$ and $M(\tilde{E}_n^{(1,1)})$ defined by the presentations given in Definitions~\ref{Def_PresD4_Yamada} and \ref{DefPres_Elliptic_En_Yamada}. Saito \cite{SaitoHigherHomotopy, Saito2ndHomotopyNonCancellative} showed that these monoids are not cancellative.  

We will show in this section that the monoids $M'(\tilde{D}_4^{(1,1)})$ and $M'(\tilde{E}_n^{(1,1)})$ defined by our presentations in Propositions~\ref{Prop_New_PresD4} and \ref{Prop_NewPres_Elliptic_En} are indeed cancellative. This is a big step towards the construction of Garside structures related to elliptic Artin monoids and groups.

In order to prove cancellativity, we will use some combinatorial techniques that were developed by Dehornoy, namely the word reversing technique. We will start by recalling this technique from \cite{DehornoyCompletePosPres}.

Let $M$ be a monoid defined by a positive presentation $\langle S\ |\ R \rangle$ with set of generators $S$ and set of positive relations $R$. Let $S^{*}$ be the free monoid generated by $S$ and let $\epsilon$ be the empty word.

\begin{defn}[Word reversing]\label{Def_WordReversing}
    Let $w$ and $w'$ be two words over $S \cup S^{-1}$. We say that we right-reverse $w$ to $w'$ in one step and write $w \curvearrowright_{r}^{(1)} w'$ if $w'$ is obtained from $w$ by
    \begin{itemize}
        \item either deleting some subword $u^{-1}u$ for $u \in S^{*} \setminus \{0\}$,
        \item or replacing some subword $u^{-1}v$ for $u,v\in S^{*} \setminus \{0\}$ by a word $v'u'^{-1}$ such that $uv'=vu'$ is a relation in $R$.
    \end{itemize}
    Let $w_0, w_1, \cdots$ be a sequence of words over $S$. We say that it is a right-reversing sequence of words if we have $w_i \curvearrowright_{r}^{(1)} w_{i+1}$ for every $i \geq 0$. Let $k \geq 1$. We write $w \curvearrowright_{r}^{(k)} w'$ if there exists a right-reversing sequence of length $k$ from $w$ to $w'$. We say that $w$ right-reverses to $w'$ and write $w \curvearrowright_{r} w'$ if $w \curvearrowright_{r}^{(k)} w'$ for some $k$.
\end{defn}

Symmetrically, we say that $w$ left-reverses to $w'$ if $w'$ is obtained from $w$ by iteratively deleting some subword $uu^{-1}$ and replacing subwords $uv^{-1}$ by words $v'^{-1}u'$ such that $v'u = u'v$ is a relation in $R$.

To every right-reversing sequence of words $w_0, w_1, \cdots$, we associate a planar graph as follows. First, we associate with $w_0$ a path labelled with the successive letters of $w_0$: we associate to every positive letter $s$ a horizontal right-oriented edge labelled $s$, and to every negative letter $s^{-1}$ a vertical down-oriented edge labelled $s$. Then, successively from left to right we represent the words $w_1, w_2, \cdots$ as follows: 
\begin{itemize}
    \item If $w_{i+1}$ is obtained from $w_i$ by replacing $u^{-1}v$ by $v' u'^{-1}$, meaning that $uv' = vu'$ is a relation of the presentation, then the graph is completed according to this relation as follows.
    
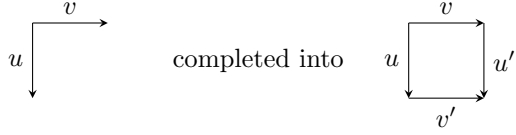
\begin{center}
    \begin{tikzpicture}
        \draw [-stealth] (0,0) to node[auto] {$v$} (1,0);
        \draw [-stealth] (0,0) to node[left] {$u$} (0,-1);

        \node[] (text) at (3,-0.5) {completed into};

        \draw [-stealth] (5,0) to node[auto] {$v$} (6,0);
        \draw [-stealth] (5,0) to node[left] {$u$} (5,-1);
        \draw [-stealth] (5,-1) to node[below] {$v'$} (6,-1);
        \draw [-stealth] (6,0) to node[right] {$u'$} (6,-1);
    \end{tikzpicture}
    \end{center}
    \item If $w_{i+1}$ is obtained from $w_i$ by deleting some subword $u^{-1}u$, then the empty word $\epsilon$ appears and the graph is completed by a non-oriented edge labelled $\epsilon$ as follows.
    
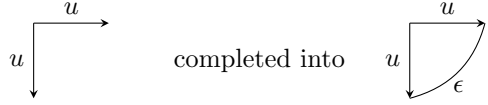
\begin{center}
    \begin{tikzpicture}
        \draw [-stealth] (0,0) to node[auto] {$u$} (1,0);
        \draw [-stealth] (0,0) to node[left] {$u$} (0,-1);

        \node[] (text) at (3,-0.5) {completed into};

        \draw [-stealth] (5,0) to node[auto] {$u$} (6,0);
        \draw [-stealth] (5,0) to node[left] {$u$} (5,-1);
        \draw [-,bend right] (5,-1) to node[below] {$\epsilon$} (6,0);
    \end{tikzpicture}
    \end{center}
\end{itemize}

A symmetric construction is associated with left-reversing. 

\begin{exmp}
    Consider the monoid presentation given in Proposition~\ref{Prop_New_PresD4}. We have that 
    \begin{center}
    \begin{tabular}{lll}
       $\underline{t_2^{-1}s_3}s_3$  & $\curvearrowright_{r}^{(1)}$ & $s_3t_2s_3^{-1}\underline{t_2^{-1}s_3}$ by using $s_3t_2s_3 = t_2s_3t_2$\\
         & $\curvearrowright_{r}^{(1)}$ & $s_3t_2\underline{s_3^{-1}s_3}t_2s_3^{-1}t_2^{-1}$ by using $s_3t_2s_3 = t_2s_3t_2$\\
         & $\curvearrowright_{r}^{(1)}$ & $s_3t_2t_2s_3^{-1}t_2^{-1}$.
    \end{tabular}
    \end{center}
    The associated graph is as follows.
    
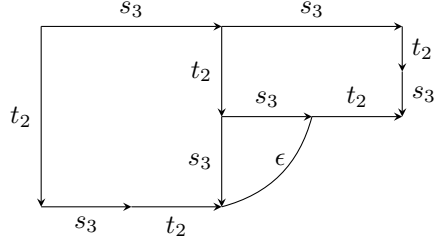
\begin{center}
    \begin{tikzpicture}[scale=1.2]
    \draw [-stealth] (0,2) to node[left] {$t_2$} (0,0);
    \draw [-stealth] (0,2) to node[above] {$s_3$} (2,2);
    \draw [-stealth] (0,0) to node[below] {$s_3$} (1,0);
    \draw [-stealth] (1,0) to node[below] {$t_2$} (2,0);
    \draw [-stealth] (2,2) to node[left] {$t_2$} (2,1);
    \draw [-stealth] (2,1) to node[left] {$s_3$} (2,0);
    \draw [-stealth] (2,2) to node[auto] {$s_3$} (4,2);
    \draw [-stealth] (4,2) to node[auto] {$t_2$} (4,1.5);
    \draw [-stealth] (4,1.5) to node[auto] {$s_3$} (4,1);
    \draw [-stealth] (2,1) to node[auto] {$s_3$} (3,1);
    \draw [-stealth] (3,1) to node[auto] {$t_2$} (4,1);
    \draw [bend right] (2,0) to node[above] {$\epsilon$} (3,1);
    \end{tikzpicture}
    \end{center}
\end{exmp}

\begin{defn}[Complemented presentations]\label{Def_Complemented}
    The positive presentation $\langle S\ |\ R \rangle$ of a monoid $M$ is said to be right-complemented if for any $x,y \in S$, there is at most one relation in $R$ of the form $x \cdots = y \cdots$ and no relations of the form $x \cdots = x \cdots$.
    The left-complemented notion is defined symmetrically.
    We say that the presentation is complemented if it is right- and left-complemented. 
\end{defn}

Notice that the presentations of the monoids $M(\tilde{D}_4^{(1,1)})$ and $M(\tilde{E}_n^{(1,1)})$ as described in Definitions~\ref{Def_PresD4_Yamada} and \ref{DefPres_Elliptic_En_Yamada} are not complemented. 

\begin{lem}\label{Lem_NewPres_Complemented}
    The presentations of the monoids $M'(\tilde{D}_4^{(1,1)})$ and $M'(\tilde{E}_n^{(1,1)})$ as described in Propositions~\ref{Prop_New_PresD4} and \ref{Prop_NewPres_Elliptic_En} are complemented.
\end{lem}

\begin{proof}
    The proof is straightforward by looking at all the relations that define the presentations.
\end{proof}

Notice that for all $u,v \in S^{*}$, $u^{-1}v \curvearrowright_{r} \epsilon$ implies that $u = v$ in the monoid. The notion of right-complete presentations that we will introduce now is when the converse holds.

\begin{defn}[Complete presentations]\label{Def_Complete}
    The positive presentation $\langle S\ |\ R \rangle$ of a monoid $M$ is said to be right-complete if for any $u,v \in S^{*}$, $u = v$ in $M$ implies that $u^{-1}v \curvearrowright_{r} \epsilon$. The notion of left-complete presentations is defined symmetrically. We say that the presentation is complete if it is right- and left-complete.
\end{defn}

The next result is due to Dehornoy~\cite{DehornoyCompletePosPres}.

\begin{prop}[The cube condition]\label{Prop_CubeCondition}
    The positive presentation $\langle S\ |\ R \rangle$ of a monoid $M$ is right-complete if and only if every triple of words $(u,v,w)$ over $S$ satisfies the right-cube condition:
    \begin{center}
    $u^{-1}ww^{-1}v \curvearrowright_{r} v'u'^{-1}$ implies $(uv')^{-1}vu' \curvearrowright_{r} \epsilon$.
    \end{center}
    A left-cube condition for left-completeness is defined symmetrically.
\end{prop}

\begin{exmp}
    We describe the right-cube condition for the triple $(s_7,t_2,s_8)$ in the monoid $M'(\tilde{E}_8^{(1,1)})$. Notice that $t_2$ commutes with $s_7$ and $s_8$, and we have $s_7s_8s_7 = s_8s_7s_8$. Let $u := s_7$, $v := t_2$, and $w := s_8$.
    
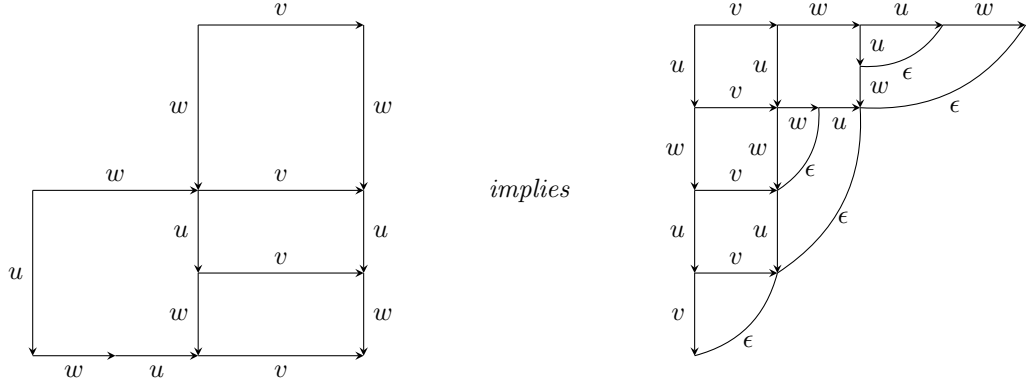
\begin{center}
    \begin{tikzpicture}[scale=1.1]
    \draw [-stealth] (0,2) to node[left] {$u$} (0,0);
    \draw [-stealth] (0,2) to node[above] {$w$} (2,2);
    \draw [-stealth] (2,4) to node[left] {$w$} (2,2);
    \draw [-stealth] (2,4) to node[above] {$v$} (4,4);
    \draw [-stealth] (0,0) to node[below] {$w$} (1,0);
    \draw [-stealth] (1,0) to node[below] {$u$} (2,0);
    \draw [-stealth] (2,2) to node[left] {$u$} (2,1);
    \draw [-stealth] (2,1) to node[left] {$w$} (2,0);
    \draw [-stealth] (2,2) to node[above] {$v$} (4,2);
    \draw [-stealth] (4,4) to node[right] {$w$} (4,2);
    \draw [-stealth] (2,1) to node[above] {$v$} (4,1);
    \draw [-stealth] (4,2) to node[right] {$u$} (4,1);
    \draw [-stealth] (2,0) to node[below] {$v$} (4,0);
    \draw [-stealth] (4,1) to node[right] {$w$} (4,0);

    \node[] (text) at (6,2) {implies};

    \draw [-stealth] (8,1) to node[left] {$v$} (8,0);
    \draw [-stealth] (8,2) to node[left] {$u$} (8,1);
    \draw [-stealth] (8,3) to node[left] {$w$} (8,2);
    \draw [-stealth] (8,4) to node[left] {$u$} (8,3);
    \draw [-stealth] (8,4) to node[above] {$v$} (9,4);
    \draw [-stealth] (9,4) to node[above] {$w$} (10,4);
    \draw [-stealth] (10,4) to node[above] {$u$} (11,4);
    \draw [-stealth] (11,4) to node[above] {$w$} (12,4);

    \draw [-stealth] (8,3) to node[above] {$v$} (9,3);
    \draw [-stealth] (9,4) to node[left] {$u$} (9,3);
    \draw [-stealth] (8,2) to node[above] {$v$} (9,2);
    \draw [-stealth] (9,3) to node[left] {$w$} (9,2);
    \draw [-stealth] (8,1) to node[above] {$v$} (9,1);
    \draw [-stealth] (9,2) to node[left] {$u$} (9,1);
    \draw [bend right] (8,0) to node[below] {$\epsilon$} (9,1);

    \draw [-stealth] (9,3) to node[below] {$w$} (9.5,3);
    \draw [-stealth] (9.5,3) to node[below] {$u$} (10,3);
    \draw [-stealth] (10,4) to node[right] {$u$} (10,3.5);
    \draw [-stealth] (10,3.5) to node[right] {$w$} (10,3);

    \draw [bend right] (9,2) to node[below] {$\epsilon$} (9.5,3);
    \draw [bend right] (9,1) to node[below] {$\epsilon$} (10,3);

    \draw [bend right] (10,3.5) to node[below] {$\epsilon$} (11,4);
    \draw [bend right] (10,3) to node[below] {$\epsilon$} (12,4);
    \end{tikzpicture}
    \end{center}
\end{exmp}

\begin{lem}\label{Lem_NewPres_Complete}
    The presentations of the monoids $M'(\tilde{D}_4^{(1,1)})$ and $M'(\tilde{E}_n^{(1,1)})$ are complete.
\end{lem}

\begin{proof}
    By Proposition~\ref{Prop_CubeCondition}, it is sufficient to show the right- and left-cube conditions for any triple of words over the generating set. Since the presentations are also homogeneous, meaning that every relation preserves the length of words, it is enough to check the cube condition for all triples of generators - and not words - (see Proposition~4.4 of \cite{DehornoyCompletePosPres}). In view of our presentations given in Propositions~\ref{Prop_New_PresD4} and \ref{Prop_NewPres_Elliptic_En}, all the cases for the triples of generators that we need to consider are already done in Section~3.2.1 of \cite{CorranPicantin}. This establishes our result.
\end{proof}

In order to conclude cancellativity of monoids, let us recall this result of Dehornoy \cite{DehornoyCompletePosPres}.

\begin{prop}[Corollary 6.2 in \cite{DehornoyCompletePosPres}]\label{Prop_ThmDehornoy_Cancel}
    Assume that the presentation $\langle S\ |\ R \rangle$ of a monoid $M$ is complete and complemented, then the monoid $M$ is cancellative.
\end{prop}

The next theorem is the main result of this section. 

\begin{thm}\label{Thm_Cancellativity}
    The elliptic Artin monoids $M'(\tilde{D}_4^{(1,1)})$ and $M'(\tilde{E}_n^{(1,1)})$ are cancellative.
\end{thm}

\begin{proof}
    By Lemma~\ref{Lem_NewPres_Complemented}, the presentations are complemented. By Lemma~\ref{Lem_NewPres_Complete}, they are also complete. Applying Proposition~\ref{Prop_ThmDehornoy_Cancel}, the elliptic Artin monoids $M'(\tilde{D}_4^{(1,1)})$ and $M'(\tilde{E}_n^{(1,1)})$ defined by the presentations of Propositions~\ref{Prop_New_PresD4} and \ref{Prop_NewPres_Elliptic_En} are cancellative.
\end{proof}

\bibliographystyle{alpha}
\bibliography{Bibliography}
\vspace*{1em}
\end{document}